\documentclass[12pt]{amsart}

\usepackage{amsthm}
\usepackage{hyperref}
\usepackage{geometry}
\usepackage{xpatch}
\usepackage{amssymb}
\usepackage{setspace}
\usepackage{amsthm}
\usepackage{enumerate}
\usepackage{commath}
\usepackage{hyperref}
\usepackage{titlesec}
\usepackage{amsmath,amssymb,amsfonts,anysize,mathrsfs,graphicx,setspace,parskip,multicol,upgreek,yfonts,tikz,physics,graphicx,xcolor,enumerate,csquotes,hyperref}

\DeclareMathOperator{\cl}{cl}

\DeclareMathOperator{\andd}{and}
\DeclareMathOperator{\Lim}{Lim}

\newtheoremstyle{mystyle}% Custom theorem style
  {}% Space above
  {}% Space below
  {\itshape}% Body font
  {}% Indent amount
  {\bfseries}% Theorem head font
  {}% Punctuation after theorem head
  {.5em}% Space after theorem head
  {}% Theorem head spec (can be left empty, meaning 'normal')

\theoremstyle{mystyle}% Applying the custom style

\newtheorem{theorem}{Theorem}% Theorem environment
\newtheorem{problem}{Problem}

\newtheorem{lemma}{Lemma}
\newtheorem{proposition}{Proposition}

\newtheorem{corollary}{Corollary}
\titleformat{\section}{\small\bfseries}{\thesection}{1em}{}
\usepackage{titlesec}

\titleformat{\subsection}{\small\bfseries}{\thesubsection}{1em}{}

\theoremstyle{definition}
\newtheorem{definition}{Definition}

\numberwithin{equation}{section}

\begin{document}

\title{$\Psi$-spaces and semi-proximality}
\author[K. Almontashery]{Khulod Almontashery} 

\address{Department of Mathematics and Statistics\\ University of Jeddah\\ Jeddah\\ Saudi Arabia}
\email{kalmontashri@uj.edu.sa}

\author[V. Rodrigues]{Vinicius de Oliveira Rodrigues}
\author[P.J. Szeptycki]{Paul J. Szeptycki}

\address{Department of Mathematics, University of São Paulo\\ São Paulo, São Carlos\\ Brazil}
\email{vinicius.rodrigues@icmc.usp.br}

\address{Department of Mathematics and Statistics\\ York University\\ Toronto, ON M3J 1P3\\Canada}
\email{szeptyck@yorku.ca}
\begin{abstract}
    We discuss the proximal game and semi-proximality in $\Psi$-spaces of almost disjoint families over an infinite countable set and $\Psi$-spaces of ladder systems on $\omega_1$. We show that a semi-proximal almost disjoint families must be nowhere MAD, anti-Luzin and characterize semi-proximality for a class of ${\mathbb R}$-embeddable almost disjoint families.  We show that a $\Psi$-spaces defined from a uniformizable ladder system is semi-proximal and a $\Psi$-space defined on a $\clubsuit^*$ sequence is not semi-proximal. Thus the existence of non-semi-proximal $\Psi$-space over a ladder system is independent of ZFC.
\end{abstract}


\maketitle
\section{Introduction}
In \cite{bell2014infinite}, Bell introduced the proximal game, an infinite two-player game played on a uniform space. The existence (or lack of) winning strategies for the players have several topological consequences for the underlying topological spaces. J. Bell invented the game to prove that certain uniform box products were normal and she found wider applications, e.g., giving an elegant alternative proof of M. E. Rudin's celebrated result that a $\Sigma$-product of metrizable spaces is collectionwise normal and countably paracompact. 

In \cite{nyikos2014proximal}, Nyikos used the proximal game to define the concept of semi-proximality. In this paper, we study semi-proximality in the realm of $\Psi$-spaces of almost disjoint families over a countable set, and those generated by ladder systems.

In this section, we review the relevant definitions and history, and summarize the new results.

\subsection{Uniform spaces}
Before we discuss the proximal game, we fix the notation for uniformities and recall some well-known facts about them. This subsection is, in no way, an introduction to this subject. For the basics on uniformities, we refer to \cite{Engelking1989}.

Given a set $X$, $\Delta_X$ is its diagonal, $\{(x, x): x \in X\}\subseteq X^2$. If $X$ is clear from the context, we simply write $\Delta$. 

If $U, V\subseteq X^2$, we let $U^{-1}=\{(y, x): (x, y)\in U\}$ and $U\circ V=\{(x, z): \exists y\, (x, y)\in U \, (y, z) \in V\}$. If $U^{-1}=U$, $U$ is said to be symmetric. We define $U^2=U\circ U$, and recursively define $U^n$ for $n \geq 1$ so that $U^{n+1}=U^n\circ U$.

\begin{definition} A uniform space is a pair $(X, \mathbb U)$ such that, for all $U, U_0, U_1\in \mathbb U$:

\begin{enumerate}
\item $\Delta_X\subseteq U\subseteq X^2$
\item $U^{-1}\in \mathbb U$
 \item There exists $V \in \mathbb U$ such that $V^2\subseteq U$.
\item $U_0\cap U_1\in \mathbb U$.
\item For all $V\subseteq X^2$, if $U\subseteq V$ then $V\in \mathbb U$.
\item[]

\end{enumerate}
The set $\mathbb U$ is then called a uniformity (on $X$). Each element of $\mathbb U$ is said to be an \textit{entourage}.
\end{definition}
If $\mathbb B\subseteq \mathbb U$ satisfies properties 1. to 4., we say it is a basis for $\mathbb U$. In that case, $\mathbb B_\uparrow$, which is defined as $\{U\subseteq X^2: \exists V\in \mathbb B\, V\subseteq U\}$, coincides with $\mathbb U$.

A uniformity $\mathbb U$ on $X$ naturally induces a topology on $X$, which we call $\tau_{\mathbb U}$. If $x \in X$ and $U \in \mathbb U$, we define $U[x]=\{y \in X: (x, y)\in U\}$. Notice that $x \in X$. Then $A\subseteq X$ is in $\mathbb \tau_{\mathbb U}$ if and only if, for every $x \in A$ there exists $U \in \mathcal U$ such that $U[x]\subseteq A$. $\tau_{\mathbb U}$ is easily seen to be a topology on $X$. The sets $U[x]$ need not to be open, however, $x$ is always in the interior of $U[x]$. Clearly, if $\mathbb B$ is a basis of $\mathbb U$, then for every $x \in X$, $\{U[x]:x \in \mathbb B\}$ is a fundamental neighborhood system at $x$. It is easily verified that $\tau_{\mathbb U}$ is Hausdorff if and only if $\Delta_X=\bigcap \mathbb U$. Moreover, the following is well-known.

\begin{lemma}Let $(X, \mathbb U)$ be a uniform space. Then:
\begin{enumerate}
    \item The collection of all symmetric, open (in $X^2$) entourages is a basis.
    \item If $U\in \mathbb U$ is open, then for every $x \in X$, $U[x]$ is open.
\end{enumerate}
\end{lemma}

We will often use the following basic lemma.

\begin{lemma}
    Let  $(X, \mathbb U)$ be a uniform space. If $U \in \mathbb U$ is symmetric and $a \in X$, then $\overline{U}[a]\subseteq 2U[a]$.
\end{lemma}
\begin{proof}
    Let $x \in \overline{V}[a]$. Then $V[x]\cap V[a]\neq \emptyset$, thus there exists $z \in X$ such that $(x, z), (a, z)\in V$. As $V$ is symmetric, $(x, a)\in 2V$, so $x \in 2V[a]$.
\end{proof}

A topological space is said to be \textit{uniformizable} if it admits a uniformity that generates its topology. It is well known (see e.g., \cite{Engelking1989}) that a topological space is uniformizable if, and only if it is completely regular, and that such uniformity is unique if, and only if it is compact.

 A uniform space $(X, {\mathbb U})$ is said to be totally bounded if for every $U \in \mathbb U$, there exists a finite $F\subseteq X$ such that $X=\bigcup_{x \in F}U[x]$. If $U$ generates a Hausdorff topology, $(X, {\mathbb U})$ is totally boundeded if and only if there is a compactification of $(X, \tau_{\mathbb U})$ such that ${\mathbb U}$ is the sub-uniformity inherited from the compactification \cite{Engelking1989}. 

\subsection{The proximal game}
The proximal game was introduced in \cite{bell2014infinite}. It is played in a uniform space $(X, \mathbb U)$ by two players: \textbf{I}, also called \textbf{entourage}, and \textbf{II}, also called \textbf{point}. The game has $\omega$ innings. At inning $0$, \textbf{I} plays some $U_0 \in \mathbb U$ and \textbf{II} responds with some $x_0 \in X$. At inning $n+1$, \textbf{I} plays some $\mathbb U_{n+1}\subseteq U_n$ in $\mathbb U$ and \textbf{II} plays some $x_{n+1}\in U_n[x_n]$. After $\omega$ steps, we say that \textbf{I} wins if, and only if either $(x_n: n \in \omega)$ converges or $\bigcap_{n \in \omega}U_n[x_n]=\emptyset$. Otherwise, \textbf{II} wins.

Formally:

\begin{definition}
    Let $(X, \mathbb U)$ be a uniform space and $\mathbb B$ be a basis for $\mathbb U$.
    \begin{itemize}
        \item A run of the proximal game on $(X, \mathbb B)$ is a sequence $((U_n, x_n): n \in \omega)$ such that, for every $n \in \omega$:
        \begin{itemize}
            \item $U_n \in \mathbb B$, $x_n \in \mathbb X$,
            \item $U_{n+1}\subseteq U_n$,
            \item $x_{n+1} \in U_n[x_n]$
        \end{itemize}
        \item A strategy for \textbf{I} is a function $\sigma: X^{<\omega}\rightarrow \mathbb B$ such that for every $n \in \omega$ and $\vec x \in X^{n+1}$, we have $\sigma(\vec x)\subseteq \sigma(\vec x|n)$.
        \item We say that a run of the proximal game $((U_n, x_n): n \in \omega)$ follows a strategy for \textbf{I}, $\sigma$, if for every $n \in \omega$, $U_n=\sigma(x_i: i<n)$.
        \item A strategy for \textbf{II} is a function $\sigma: \mathbb B^{<\omega}\rightarrow X$ such that for every $n \in \omega$ and $\vec U\in \mathbb B^{n+2}$, $\sigma(\vec U)\in \vec U_n[\sigma(\vec U|_{n+1})]$.
        \item We say that a run of the proximal game $((U_n, x_n): n \in \omega)$ follows a strategy for \textbf{I} (respectively, \textbf{I}), $\sigma$, if for every $n \in \omega$, $U_n=\sigma(x_i: i<n)$ (respectively, $x_n=\sigma(U_i: i\leq n)$).        
         \item We say that a strategy for \textbf{I} (respectively, \textbf{II}) is a winning strategy if \textbf{I} (respectively, \textbf{II}) wins every game which follows the strategy.
         \item We say that the uniformity $(X, \mathbb U)$ is proximal if there exists a winning strategy on $(X, \mathbb U)$.
         \item We say that the uniformity $(X, \mathbb U)$ is semi-proximal if there is no winning strategy for \textbf{II} in the proximal game for $(X, \mathbb U)$.       
    \end{itemize}
    \end{definition}

    Some technical notes: our definition of strategy does not exclude from its domain invalid partial games, e.g. a strategy for \textbf{II} is defined for all pairs $(U_0, U_1)\in \mathbb B^2$, even those for which $U_1 \not \subseteq U_0$. 
    
    We did not exclude such sequences from the domains of the strategies because that would make their definitions much more cumbersome. However, to verify that a strategy is winning, we only need to care about ``valid partial plays'' as we only need to look at valid games which follow the strategies. Thus, in a concrete context, when defining a strategy one may only care about the ``valid partial games'', and define it for ``invalid partial plays'' arbitrarily. 

An open cover $\mathcal{U}$ of $X$ is called \textit{normal} if there exists a sequence of open covers $\{\mathcal{V}_n : n \in \omega\}$ such that $\mathcal{V}_0 = \mathcal{U}$ and each $\mathcal{V}_n$ is a star refinement of $\mathcal{V}_{n-1}$. Uniformities can be described equivalently using normal covers. Specifically, if $\mathcal{A}$ is the corresponding family of normal covers corresponding to the uniformity ${\mathfrak U}$, the proximal game can be reformulated as follows:  

\begin{itemize}
    \item In inning $0$, Player I selects an open cover $A_0 \in \mathcal{A}$, and Player II chooses a point $x_0 \in X$.  
    \item In inning $n+1$, Player I picks an $A_{n+1} \in \mathcal{A}$ that refines $A_n$, and Player II chooses a point $x_{n+1} \in \operatorname{St}(x_n, A_n)$, where $\operatorname{St}(x_n, A_n) = \bigcup \{U \in A_n : x_n \in U\}$.
\end{itemize}  

Player I wins if either  $
\bigcap_{n < \omega} \operatorname{St}(x_n, A_n) = \emptyset
$
or the sequence $(x_n : n < \omega)$ converges.

    Proximal and semi-proximal topological spaces are defined as follows:

    \begin{definition}
        A topological space $(X, \tau)$ is said to be proximal if there exists a uniformity $\mathbb U$ for $X$ such that $(X, {\mathbb U})$ is proximal and $\tau=\tau_\mathbb U$ (such a uniformity is said to be {\em compatible}). If, additionally, $\mathbb U$ is totally bounded, $(X, \tau)$ is said to be \textbf{totally proximal}.
        
        Similarly, a topological space $(X, \tau)$ is said to be semi-proximal if there exists a compatible uniformity $\mathbb U$ for $X$ such that $(X, \mathbb U)$ is semi-proximal. If, additionally, $\mathbb U$ is totally bounded, $(X, \tau)$ is said to be \textbf{totally semi-proximal}.
    \end{definition}

We refer the reader to \cite{KhThesis} for more about total proximality and total semi-proximality.

    J.\ Bell proved that every proximal space is collectionwise normal, countably paracompact and a $W$-space (thus, $\alpha_2$ and Fréchet) \cite{bell2014infinite}. She also proved that the uniformities induced by metrics are proximal and that proximality is preserved by $\Sigma$-products. These results, together, give an alternative proof of Rudin's celebrated theorem stating that every $\Sigma$-product of metrizable spaces is collectionwise normal and countably paracompact.

    Semi-proximality was subsequently defined by Nyikos in \cite{nyikos2014proximal}. He proved that semi-proximal spaces are $w$-spaces (thus, $\alpha_2$ and Fréchet) and asked whether semi-proximal spaces are normal. This question was answered in the negative in \cite{AS} - we will comment further on this in the next subsection.

    It is easily seen that the proximal game on a given uniformity $\mathbb U$ is equivalent when Player I is restricted to playing from a uniform base.
    %- that is:

%\begin{lemma}
 %       Let $(X, \mathbb U)$ be a uniform space and $\mathbb B$ be a basis of $\mathbb U$. Then $(X, \mathbb U)$ is proximal (semi-proximal) if, and only if \textbf{I} has a winning strategy in the proximal game for $(X, \mathbb B)$ (\textbf{II} has no winning strategy in the proximal game for $(X, \mathbb B)$).
  %  \end{lemma}

 \subsection{Almost disjoint families and ladder systems}
    An almost disjoint family on a infinite countable set $N$ is a collection $\mathcal A$ of infinite subsets of $N$ such that for every two distinct $a, b \in \mathcal A$, $a\cap b$ is finite. It is said to be a maximal almost disjoint family (MAD family) if it is maximal with respect to containment among all almost disjoint families over $N$. The free ideal generated by $\mathcal A$ is $\mathcal I(\mathcal A)=\{X\subseteq N: \exists \mathcal A' \in [\mathcal A]^{<\omega}\, X\subseteq^* \bigcup \mathcal A\}$. We define $\mathcal I^+(\mathcal A)=\mathcal P(N)\setminus \mathcal I(\mathcal A)$.

    Given an almost disjoint family $\mathcal A$ over $N$, where $N$ is infinite and countable and $N\cap [N]^\omega=\emptyset$, one defines the Isbell-Mrówka space $\Psi(\mathcal A)$, also known as the $\Psi$-space of $\mathcal A$:  $\Psi({\mathcal A})=N\cup \mathcal A$ topologized so $N$ is open and discrete, and for every $a \in \mathcal A$, the set $\{\{a\}\cup a\setminus F: F\in [N]^{<\omega}\}$ is a local basis at $a$. This is the strongest topology which makes $N$ open and discrete and for which every $a \in \mathcal A$ is the range of a sequence converging to the point $a$.

    Isbell-Mrówka spaces are Tychonoff, locally compact, zero-dimensional, separable and Moore. Combinatorial properties of $\mathcal A$ are closely related to topological properties of $\Psi(\mathcal A)$. For instance, $\mathcal A$ is maximal if and only if $\Psi(\mathcal A)$ is pseudocompact. For that reason, the construction of almost disjoint families with specific combinatorial properties have often been employed in the construction of examples in general topology.
    For instance, the study of normality in Isbell-Mrówka spaces is related to Moore's problem on the existence of a normal non-metrizable Moore space - which is known as the Normal Moore Space Problem \cite{HRUSAK2007111}. In fact, the existence of a normal $\Psi$-space is equivalent to the existence of a normal non-metrizable separable Moore space \cite{tall1977set}. We refer to \cite{Hrusak2014} as a survey on $\Psi$-spaces.

    It is easy to classify the proximal $\Psi$-spaces: A $\Psi$-space is metrizable if, and only if it is countable and since all $\Psi$-spaces are regular and Moore, $\Psi$-spaces are collectionwise normal if and only if they are countable. As proximal spaces are collectionwise normal, a $\Psi$-space is proximal if and only if it is countable. Moreover, the one-point compactification of countable $\Psi$-spaces is also metrizable, so countability also characterizes their total proximality.

    The problem of characterizing the semi-proximal $\Psi$-spaces is more complicated, and is the main problem we consider in this paper. As as shorthand, we define:

    \begin{definition}
        We say that an almost disjoint family is (totally) semi-proximal if its $\Psi$-space is (totally) semi-proximal.
    \end{definition}

A particularly important class of almost disjoint families are those obtained as a set of branches in the trees $2^{<\omega}$ or $\omega^{<\omega}$. If $x \in 2^\omega$, we define $a_x=\{x|n: n \in \omega\}$, and, for $X\subseteq 2^\omega$, $\mathcal A_X=\{a_x: x \in X\}$. The set $\mathcal A_X$ is easily seen to be an almost disjoint family on $2^{<\omega}$. An almost disjoint family of this type is said to be an almost disjoint family of (binary) branches. The topological properties of $X$ often characterize topological properties of $\Psi(\mathcal A_X)$. For instance, the former is a $Q$-set (respectively, $\sigma$-set, $\lambda$-set, weak $\lambda$-set) if, and only if the latter is normal (respectively, almost-normal, pseudonormal, strongly $\aleph_0$-separated, see e.g. \cite{de2023special}). 

In \cite{AS}, the first and third authors proved that a branching family $\mathcal A_X$ is semi-proximal in case $X$ does not contain a copy of the Cantor set and this was used to give a ZFC example of a semi-proximal space which is not normal. In Section 2 we establish the converse, thus, characterizing the semi-proximality of almost disjoint families of branches. We also generalize the preceding implication to $\mathbb R$-embeddable almost disjoint families.

        In Section 3, we give an insight into how close semi-proximal $\Psi$-spaces are to being normal. It is known that if $\Psi(\mathcal A)$ is normal, then $\mathcal A$ is nowhere MAD and anti-Luzin. We will show that semi-proximal almost disjoint families are also nowhere mad and anti-Luzin.

 An important class of examples of almost disjoint families of countable infinite subsets of $\omega_1$ are the so-called ladder systems. A ladder system is a family $\mathcal L=(L_\alpha: \alpha\in \Lim)$, where $\Lim$ is the set of all countable limit ordinals in $\omega_1$, and $L_\alpha$ is a cofinal subset of $\alpha$ of type $\omega$. 
 The concept of uniformizability of a ladder system (defined below) was introduced by S.\ Shelah in his celebrated solution to Whitehead's problem \cite{Shelah1980}.

Given a ladder system $\mathcal L$, one naturally defines a $\Psi$-space $\Psi(\mathcal L)$: the base space is $(\omega_1\times \{0\})\cup \Lim\times \{1\}$, where $\omega_1\times \{0\}$ is open and discrete, and a local basis for $(\alpha, 1)$ is $(\{(\alpha, 1)\}\cup (L_\alpha\setminus \xi)\times\{0\}: \xi<\alpha\}$. These spaces are also Tychonoff, locally compact and zero-dimensional. As before, we say that $\mathcal L$ is (totally) semi-proximal if $\Psi(\mathcal L)$ is.

In Section 4, we show that the existence of a not semi-proximal ladder system $\mathcal L$ is independent of ZFC by showing that the ones given by $\clubsuit^*$-sequences are not semi-proximal, and those that are uniformizable (which is consistently all of them) are semi-proximal.
        
\section{Semi-proximal almost disjoint families and \texorpdfstring{$\mathbb R$}{R}-embeddability}

As mentioned in the previous section, in\textit{} \cite{AS}, it was proven that $\mathcal A_X$ is semi-proximal in case $X$ does not contain a copy of the Cantor set. We aim to establish the converse.

\begin{proposition} \label{CantorCopy} Let $X\subseteq 2^\omega$. The following are equivalent:

\begin{enumerate}[a)]
    \item $X$ contains no perfect\footnote{$C\subseteq X$ is perfect if and only if it is closed, non-empty and has no isolated points} subsets.
    \item $\mathcal A_X$ is semi-proximal.
    \item $\mathcal A_X$ is totally semi-proximal.
\end{enumerate}
\end{proposition}
\begin{proof}By \cite[Theorem 1]{AS}, if $X$ has no copy of the Cantor set, then $\mathcal A_X$ is totally semi-proximal.  Thus, it remains to show that if $X$ has a perfect set (equivalently, a copy of the Cantor set), then $\Psi(\mathcal A_X)$ is not semi-proximal.

Let $\mathbb U$ be a uniform structure for $\Psi(\mathcal A_X)$.
     Fix $C\subseteq X$ a homeomorphic copy of $2^\omega$. We show that \textbf{II} has a winning strategy $\sigma$ for the proximal game for $(X, \mathbb B)$, where $\mathbb B$ is the collection of all open, symmetric elements of $\mathbb U$.
     
     We construct mappings $\sigma, D: \bigcup_{k\geq 0}\{(U_0, \dots, U_k)\in \mathbb B^{k+1}: U_k\subseteq \dots \subseteq U_0\}\longrightarrow \mathcal P(C)$ by recursion on the length of the sequences, $k$. These mappings will satisfy, for $k\in \omega$ and every $\subseteq$-decreasing $(U_0, \dots, U_{k})\in \mathbb B^{k+1}$:

   \begin{enumerate}[1)]
   \item$D(U_0, \dots, U_{k})\subseteq D(U_0, \dots, U_{k-1})$ if $k>0$.
   \item$D(U_0, \dots, U_{k})\approx 2^\omega$.
   \item $\mathcal A_{D(U_0, \dots, U_{k})}\subseteq U_k[\sigma(U_0, \dots, U_k)]$.
         \item $\sigma(U_0, \dots, U_k) \in \mathcal A_{D(U_0, \dots, U_k)}$.
        \item $\sigma(U_0, \dots, U_k)\neq \sigma(U_0, \dots, U_{k-1})$ if $k>0$.

    \end{enumerate}
    
        We show how to define such a $\sigma$ and a $D$ recursively.\\

        Basis: for $U \in \mathbb B$, fix $W\in \mathbb B$ such that $4W\subseteq  U$. As $2^{<\omega}$ is dense and $W$ is symmetric, $\{W[s]: s \in 2^{<\omega}\}$ covers $\mathcal A_C$. Since each $W[s]$ is open, $C=\bigcup\{\{x\in C:\forall n \geq m\, x|n\in W[s]\}: m \in \omega, s \in 2^{<\omega}\}$. Also, it is easily verified that each $\{x\in C:\forall n \geq m\, x|n\in W[s]\}$ is closed, thus by Baire's Category Theorem there exists $s \in 2^{<\omega}$ and $m \in \omega$ such that $\{x\in C:\forall n \geq m\, x|n\in W[s]\}$ has nonempty interior (in the topology of $C$), thus there exists a copy of the Cantor set $D(U)$ inside of it. It follows that $\mathcal A_{D(U)}\subseteq \cl_{\Psi(\mathcal A_X)} W[s]\subseteq 2W[s]$.
        
        Choose $\sigma(U) \in \mathcal A_{D(U)}$. Then $\mathcal A_{D(U)}\subseteq 2W[s]\subseteq 4W[\sigma(U)]\subseteq U[\sigma(U)]$.\\
        
         Inductive step: Assume $k\geq 0$ and that we have defined $\sigma(U_0, \dots, U_k)$ and $D(U_0, \dots, U_k)$, where $(U_0, \dots, U_k)$ is a decreasing sequence in $\mathbb B$.  Let $U_{k+1}\in \mathbb B$ be such that $\mathcal U_{k+1}\subseteq U_k$. We must define $\sigma(U_0, \dots, U_{k+1})$ and $D(U_0, \dots, U_{k+1})$. This will be similar to the basis step, but we sketch it for the sake of completeness.

                 Fix $W\in \mathbb B$ such that $4W\subseteq  U_{k+1}$. As before, there exists $m\in \omega$ and $s \in 2^{<\omega}$ such that $\{x\in D(U_0, \dots, U_k):\forall n \geq m\, x|n\in W[s]\}$         has nonempty interior, thus there exists a copy of the Cantor set $D=D(U_0, \dots, U_{k+1})$ inside of it. It follows that $\mathcal A_{D}\subseteq \cl_{\Psi(\mathcal A_X)} W[s]\subseteq 2W[s]$.
        
        Choose $\sigma(U_0, \dots, U_{k+1}) \in \mathcal A_{D}\setminus\{\sigma(U_0, \dots, U_{k})\}$. Then $\mathcal A_{D}\subseteq 2W[s]\subseteq 4W[\sigma(U_0, \dots, U_{k+1})]\subseteq U[\sigma(U_0, \dots, U_{k+1})]$.\\
   
        To see that $\sigma$ defines a winning strategy, assume $(U_0, x_0, U_1, x_1, \dots)$ is a play of the game which follows $\sigma$. Then $(x_k)_{k \in \omega}$ does not converge as it is a non eventually constant sequence in the discrete closed set $\mathcal A_X$. It remains to see then that $\bigcap_{k \in \omega} U_k[x_k]\neq \emptyset$. For each $k \in \omega$, let $D_k=D(U_0, \dots, U_k)$.
            
    As $(D_k: k \in \omega)$ is a decreasing sequence of compact subsets of $2^\omega$, thus $\bigcap_{k \in \omega} D_k\neq \emptyset$, thus $\emptyset\neq \bigcap_{k \in \omega} \mathcal A_{D_k}\subseteq \bigcap_{k \in \omega} U_k[x_k]$.

\end{proof}
Next, we give a generalization of the aforementioned result from \cite{AS} to $\mathbb R$-embeddable almost disjoint families:
    \begin{definition}
        An almost disjoint family (over $N$) $\mathcal A$ is said to be $\mathbb R$-embeddable iff there exists a continuous $f:\Psi(\mathcal A)\rightarrow \mathbb R$ such that $f|_\mathcal A$ is injective.
    \end{definition}

    Notice that the existence of such an $f$ is equivalent to the existence of $f:N\rightarrow \mathbb R$ such that for every $a \in A$, $f[a]$ is the range of a real sequence converging to some $x_a$ so that $x_a\neq x_b$ whenever $a, b$ are distinct members of $\mathcal A$. Moreoever, the existence of such an embedding is equivalent to the existence of an embedding that is 1-1 on all of $\Psi(\mathcal A)$ (see \cite{Hrusak2014}).

    Easy examples of $\mathbb R$-embeddable families may be obtained as follows: each $x \in \mathbb R\setminus\mathbb  Q$, let ${q_x(n): n \in \omega}$ be the range a sequence of rationals converging to $x$. In this example, $f:\mathbb Q\rightarrow \mathbb R$ may be chosen as the inclusion.
        
    Equivalently, $\mathbb R$ may be substituted by $2^\omega$ in the preceding definition \cite[Lemma 2]{guzman2021mathbb}. 
    Thus, the mapping $f:2^{<\omega}\rightarrow 2^\omega$ given by $f(s)=s^\frown\mathbf{\vec 0}$ shows that all branching families are $\mathbb R$-embeddable. Given $X\subseteq A$, this induces a continuous map $\bar f:\Psi(\mathcal A_X)\rightarrow 2^\omega$ containing $f$ such that $\bar f(a_x)=x$ for every $x \in X$. However, not all $\mathbb R$-embeddable families are equivalent to a branching family \cite[Remark 5]{guzman2021mathbb}.

    First, we state the following technical lemma, which will also be used in Section 4 where we study semi-proximality of spaces defined from ladder systems We postpone its proof to the end of this section.

 \begin{lemma}\label{technicalMachinery}
    Let $X$ be a topological space, $Y\subseteq X$,  $H:Y\rightarrow 2^\omega$ be an injective function and $(\mathcal A_{k,F}: k \in \omega, F \in [Y]^{<\omega})$ be a family of clopen partitions of $X$  such that:

    \begin{enumerate}[1)]
            \item The uniformity $\mathbb U$ generated by this family of clopen partitions generates the topology of $X$.
        \item For each $F\in [Y]^{<\omega}$ and $k \in \omega$, 
        $
        \mathcal A_{k,F} = \mathcal B_{k,F} \cup \{A_{k, F}^0,A_{k, F}^1\}
        $,
        such that the elements of $\mathcal B_{k,F}$ are proximal subspaces (in the uniformity induced by $\mathbb U$), $F\subseteq \bigcup \mathcal B_{k,F}$ and, for $i\in 2$, $H(A_{k,F}^i\cap Y)\subseteq \{g \in 2^\omega: g(k)=i\}$.%\quad \text{and}\quad F\cap A_{j, T}^i=\emptyset$$

        \item For every $g \in 2^\omega$ and every $\subseteq$-increasing sequence $(F_k: k \in \omega)$ in $[Y]^{<\omega}$, $\bigcap_{k \in \omega}A_{k,F_k}^{g(k)}\subseteq Y$.
    \end{enumerate}

    Then if $H[Y]$ does not contains a copy of the Cantor set, $X$ is semi-proximal
\end{lemma}

Before we prove Lemma \ref{technicalMachinery}, we show how to apply it to $\Psi$-spaces.
    
\begin{proposition}Let $\mathcal A$ be an almost disjoint family and $\phi: \Psi(\mathcal A)\rightarrow 2^\omega$ be a continuous function such that $\phi|\mathcal A$ is injective. If $\phi[\mathcal A]$ does not contain a perfect set, then it is totally semi-proximal.
\end{proposition}

\begin{proof}
    Let $X=\Psi(\mathcal A)$. We can assume that $N=\omega$. Let $Y=\mathcal A$ and $H=\phi|_Y$.

    For each $n \in \omega$ and $i \in 2$, let $A_k^i=\{x \in X: \phi(x)(k)=i\}\setminus k$. Clearly, each $A_k^i$ is a clopen subset of $X$. For each $k \in \omega$ and $F \in [\mathcal A]^{<\omega}$, fix $k_F\geq k$ so that $$\mathcal B_{k,F}=\{\{a\}\cup (a\setminus k_F: a \in F\}\cup \{\{m\}: m< k_F\}$$ is pairwise disjoint - for clarity, $\mathcal B_{k,\emptyset}=\emptyset$. Observe that each $\mathcal B_{k,F}$ is a finite collection of clopen sets of $X$. For each $k \in \omega$, let $A_{k,F}^i = A_k^i\setminus \bigcup \mathcal B_{k,F}$ and define $\mathcal A_{k,F}=\{A_{k,F}^i: i<2\}\cup \mathcal B_{k,F}$. Each $\mathcal A_{k,F}$ is a clopen partition of $X$. 
    
    Let $\mathbb U$ be the uniformity generated by this family of clopen partitions. This uniformity is easily seen to be totally bounded and to generate the topology of $\Psi(\mathcal A)$. As $H[Y]$ does not contain a copy of the Cantor set, Lemma \ref{technicalMachinery} shows that $\mathbb U$ is semi-proximal.
\end{proof}

We do not know if the converse holds, but we conjecture that an ${\mathbb R}$-embeddable almost disjoint family is semiproximal if and only if the range of some (every) embedding contains no uncoutable perfect subset. 

\begin{problem}
    Assume $\mathcal A$ is an almost disjoint family such that there exists a continuous function $\phi:\Psi(\mathcal A)\rightarrow 2^\omega$ whose range contain no Cantor set. Is $\mathcal A$ semi-proximal?
\end{problem}

This question raises the following natural question: 

\begin{problem}\label{2.6}
    Is there an almost disjoint family $\mathcal A$ for which there exists two embeddings $\phi_0, \phi_1:\Psi(\mathcal A)\rightarrow \mathbb R$ so that $\phi_0[\mathcal A]$ contains a copy of the Cantor set, and $\phi_1[\mathcal A]$ does not?
\end{problem}

We end this section with the proof of Lemma \ref{technicalMachinery}.

\begin{proof}
   Let ${\mathbb U}$ be the uniformity consisting of entourages of the form $\bigcup\{U^2:U\in \mathcal A_{k, F}\}$ where $k \in \omega$ and $F\subseteq Y$ is finite. In what follows we play the open cover version of the proximal game where Player \textbf{I} will play such clopen partitions of the space. Let $\sigma$ be a strategy for Player \textbf{II} in the proximal game on $(X, {\mathbb U})$.
   
We define a counter-strategy for Player \textbf{I}, called  {\it the plain strategy}. The plain strategy for Player \textbf{I} is to play ${\mathcal A}_k = \mathcal{A}_{k,\emptyset}$ at inning $k$. A \textit{finite modification} of the plain strategy for Player \textbf{I} is to play partitions of the form ${\mathcal A}_{k,F}$.

We will now define, for each $g\in 2^\omega$, runs of the game $P_g$ following $\sigma$, so that if Player \textbf{II} wins each of these plays of the game, then a copy of the Cantor set would be embedded in $H[Y]$. We first define a play of the game for the constant $0$ function, $\langle \overline{0}\rangle$:

  Let $F_{\langle \overline{0}\rangle|{k+1}} = \emptyset$, for all $k\geq 0$.  In inning 0,  Player \textbf{I} chooses ${\mathcal A}_{0}$, and Player \textbf{II} chooses $\sigma({\mathcal A}_{0}) = x_{{\langle 0\rangle}}$, which gives an initial play of the game denoted by $P_{\langle0\rangle}$.
Extend it to a full play of the game where Player \textbf{I} uses the unmodified plain strategy and Player \textbf{II} uses $\sigma$ to obtain 
 $
 P_{\langle\overline{ 0}\rangle} = P_{\langle0\rangle} \smallfrown({ \mathcal A}_1, x_{\langle 00\rangle},\ldots,{ \mathcal A}_k, x_{\langle \overline{0}\rangle| k+1},\ldots).
 $
 If either
\begin{enumerate}[(1)]
    \item  there exists $k\geq 0$ such that $x_{\langle \overline{0}\rangle| k+1}\in B$, where $B\in \mathcal{B}_k$, or
    \item $\bigcap_{k\in\omega} \,\,\text{St}( x_{\langle \overline{0}\rangle| k+1}, {{\mathcal A}_k}) =\emptyset$,
\end{enumerate} 
then $\sigma$ can be defeated. Indeed, if (1) occurs, then Player \textbf{II} is forced to play inside a clopen proximal subspace, so Player \textbf{I} would be able to win by using a winning strategy from that point on. So we assume that (1) and (2) fail. Then for each $k$ there is $g_{\langle \overline{0}\rangle| k+1}\in 2^\omega$ such that $x_{\langle \overline{0}\rangle| k+1} \in A^{g_{\langle \overline{0}\rangle| k+1}(k)}_{k, \emptyset}$. Since the plain strategy was employed, there exists $y_{\langle{0}\rangle}\in \bigcap_{k\in\omega} \,\,\text{St}( x_{\langle \overline{0}\rangle| k+1}, {{\mathcal A}_k}) = \bigcap_{k\in\omega} A_{k, \emptyset}^{g_{\langle \overline{0}\rangle| k+1}(k)}.$

  Now we use ${y_{\langle{0}\rangle}}$ to define another play of the game corresponding to the branch $\langle 1\overline{0}\rangle = (1,0,0, \ldots)$ in $2^\omega$. Let  $F_{\langle 1\rangle} = F_{\langle 1\overline{0}\rangle| n+1} = \{{y_{\langle 0\rangle}}\} $, for all $n> 0$.  
In inning 0,  Player \textbf{I} uses the plain strategy modified by $F_{\langle 1\rangle}$  and chooses $ { \mathcal A}_{0, F_{\langle 1\rangle}}$ while Player \textbf{II} chooses $x_{\langle 1\rangle}$. It gives an initial play of the game
$
P_ {\langle 1\rangle} = ({{\mathcal A}_{0, F_{\langle 1\rangle}}},x_{\langle 1\rangle}).
$
We extend it to a full play of the game with Player \textbf{I} using the plain strategy only modified by $F_ {\langle 1\rangle}$ and Player \textbf{II} using $\sigma$:
$$
P_ {\langle 1\overline{ 0}\rangle} =  P_{\langle1\rangle}\smallfrown({{ \mathcal A}_{1, F_{\langle 10\rangle}}}, x_{\langle 10\rangle},\ldots,{ \mathcal A}_{k, F_{\langle 1\overline{0}\rangle| k+1}}, x_{\langle 1\overline{0}\rangle| k+1},\ldots).
$$ 
 Again, if either
\begin{enumerate}[(1)]
    \item  there exists $k\geq 0$ such that $x_{\langle 1\overline{0}\rangle| k+1}\in B$, where $B\in \mathcal{B}_{k, F_{\langle1\overline{0}\rangle| k+1}}$, or
    \item $\bigcap_{k\in\omega} \,\,\text{St}( x_{\langle 1\overline{0}\rangle| k+1}, {{\mathcal A}_{k, F_{\langle1\overline{0}\rangle| k+1}}}) =\emptyset$,
\end{enumerate} 
then $\sigma$ can be defeated, so there exists $y_{\langle 1\rangle}\in   Y$ such that 
$$
{y_{\langle{1}\rangle}}\in \bigcap_{k\in\omega} \,\,\text{St}( x_{\langle 1\overline{0}\rangle| k+1}, {{\mathcal A}_{k, F_{\langle1\overline{0}\rangle| k+1}}}) = \bigcap_{k\in\omega} (A_{k, F_{\langle 1\rangle}}^{g_{\langle 1\overline{0}\rangle| k+1}(k)})$$  
Note that $y_{\langle 1\rangle}\neq y_{\langle 0\rangle}$ by 2) of the statement as the latter is in $F_{\emptyset}$, so there exists a minimum $m_\emptyset\in\omega$ such that, in inning $m_\emptyset$,  
$g_{\langle \overline{0}\rangle| m_\emptyset+1}\neq g_{\langle 1\overline{0}\rangle\rangle| m_\emptyset+1}$ which means that $$H(y_{{\langle 0\rangle}})( m_\emptyset) \neq H(y_{{\langle 1\rangle}})( m_\emptyset) \quad \text{and} \quad H(y_{{\langle 0\rangle}})| m_\emptyset = H(y_{{\langle 1\rangle}})| m_\emptyset.$$

  Let $n>1$ and assume we have defined $P_s$, ${y_s}$ and $F_{s}$, for every $s\in 2^{\leq n}$,  $m_t\in\omega$ for every $t\in 2^{< n}$, and we have also defined $F_{s^\smallfrown\overline{0}| k}$ for all $k>|s|$ and $g_{s^\smallfrown\langle \overline{0}\rangle| k}(k)$ for all $k$ such that:
\begin{enumerate} 
\item $P_s =({{\mathcal A}_{0,F_{s | 1}}}, x_{s | 1},\ldots, {{\mathcal A}_{{n-1}, F_{s| n}}}, x_{s| n })$ which is    a partial play of the game following $\sigma$. Thus, if $s$ extends $t$, then $P_s$ extends $P_t$.

\item  $P_{s^\smallfrown\langle \overline{0}\rangle} = P_s^\smallfrown ({{\mathcal A}_{n, F_s}}, x_{s^\smallfrown\langle \overline{0}\rangle| n+1},{{\mathcal A}_{{n+1},F_s}}, x_{s^\smallfrown\langle \overline{0}\rangle| n+2},\ldots)$ is a branch of a play of the game corresponding to $s^\smallfrown\langle\overline{0}\rangle$. 

\item $F_s = \{{y_t} : t\subset s\}$. And for $k>|s|$,  $F_{s^\smallfrown\overline{0}| k} = F_s$.

\item $g_{s^\smallfrown\langle \overline{0}\rangle| k+1}(k)\in \{0,1\}$ such that $x_{s^\smallfrown \langle\overline{0}\rangle| k+1}\in  A_{k, F_{s}}^{g_{s^\smallfrown\langle \overline{0}\rangle| k+1}(k)}$.

\item ${y_s}\in \bigcap_{k\in\omega} A_{k, F_s}^{g_{s^\smallfrown\langle \overline{0}\rangle| k+1}(k)}$.

\item  $m_t$ satisfies the following:
\begin{enumerate}
    \item $H(y_{{t^\smallfrown 0}}) ( m_t) \neq H(y_{t^\smallfrown 1}) ( m_t )$, and  $H(y_{t^\smallfrown 0}) | m_t = H(y_{t^\smallfrown 1}) |  m_t $. 
    \item  $H(y_{r})| m_r = H(y_{ t})| m_r$ if $t\in[r]\cap2^{<n}$. 
    \item $H(y_{ r})| m_{\Delta(r,t)} = H(y_{ t})| m_{\Delta(r,t)} \,\,\textnormal{and}\,\, H(y_{ r})(m_{\Delta(r,t)}) \neq H(y_{ t})(m_{\Delta(r,t)})$ if $r$ and $t$ are not comparable in $2^{<n}$, where $\Delta(r,t)$ is the maximum in $2^{<n}$ which both $r$ and $t$ extend.
\end{enumerate}
\end{enumerate}
Now, for every $s\in 2^{n}$, define $F_{s^\smallfrown 0} =  F_s$, and $F_{s^\smallfrown0^\smallfrown\langle\overline 0 \rangle| k+1} = F_s$, for all $k> n$.  
 In inning $n$, let Player \textbf{I} use the plain strategy modified by $F_{s}$ and Player \textbf{II} uses $\sigma$, then the partial play of the game is $P_{s^\smallfrown 0} = P_s^\smallfrown ({{\mathcal A}_{n, F_{s}}}, x_{s^\smallfrown 0})$. We extend it to a full play of the game with the plain strategy only modified by $F_{s}$,
 $
 P_{s^\smallfrown\langle0 \overline{0}\rangle} = P_{s^\smallfrown 0}\smallfrown ({{\mathcal A}_{{n+1},F_s}}, x_{s^\smallfrown\langle \overline{0}\rangle| n+2}\ldots).
 $
  Therefore,    the play of the game, corresponding to $s^\smallfrown 0$, is equal to the one that corresponds to $s$, $P_{s^\smallfrown\langle \overline{0}\rangle}$. So, let $y_{s^\smallfrown 0} = y_s$.

   Define $F_{s^\smallfrown 1} =  F_s\cup\{{y _s}\}$ and
$F_{s^\smallfrown\langle1\overline{0}\rangle| k+1} = F_{s^\smallfrown 1}$, for $k>n$. In inning $n$, let Player I use the plain strategy modified by $F_{s^\smallfrown 1| k+1}$ at inning $k\leq n$ against $\sigma$, then the partial play of the game is $P_{s^\smallfrown 1} = P_s^\smallfrown ({{\mathcal A}_{n, F_{s^\smallfrown 1}}}, x_{s^\smallfrown 1} )$. We extend it to a full play of the game with the plain strategy modified by $F_{s^\smallfrown 1}$ at inning $k>n$, 
 $
 P_{s^\smallfrown\langle1 \overline{0}\rangle} = P_{s^\smallfrown 1}\smallfrown ({{\mathcal A}_{{n+1}, F_{s^\smallfrown 1}}}, x_{s^\smallfrown\langle1 \overline{0}\rangle| n+2},\ldots).
 $
As in the base case of the construction, if either
\begin{enumerate}[1)]
    \item there exists $k\geq 0$ such that $x_{\langle s^\smallfrown 1\overline{0}\rangle| k+1}\in B$, where $B\in \mathcal{B}_{k, F_{s^\smallfrown\langle 1\overline{0}\rangle| k+1}}$, or
  
    \item $\bigcap_{k\in\omega} \,\,\text{St}( x_{s^\smallfrown\langle 1\overline{0}\rangle| k+1}, {{\mathcal A}_{k, F_{s^\smallfrown\langle 1\overline{0}\rangle| k+1}}}) =\emptyset$
\end{enumerate} 

then, there is a play of the game where the strategy $\sigma$ is defeated. Otherwise, for each $s\in 2^n$ there exists $y_{s^\smallfrown 1}\in Y$ such that 
$$
{y_{s^\smallfrown 1}}\in \bigcap_{k\in\omega} \,\,\text{St}( x_{s^\smallfrown\langle1 \overline{0}\rangle| k+1}, {\mathcal A}_{k, F_{s^\smallfrown\langle 1\overline{0}\rangle| k+1}})
$$ 

Since ${y _s}\in F_{s^\smallfrown 1}$, then  $y_{s^\smallfrown 0} = y _s  \neq y_{s^\smallfrown 1}$, so there exists a minimum $m_{s}$ such that in inning $m_s$, $g_{s^\smallfrown\langle  \overline{0}\rangle| m_{s}+1})\neq g_{s^\smallfrown \langle 1\overline{0}\rangle| m_{s}+1}$. Hence, $H(y_{ s})( m_s) \neq H(y_{s^\smallfrown 1})(m_s)$ and $H(y_{s})| m_s = H(y_{s^\smallfrown 1})| m_s$.
 Then, the elements of $\{ y_s: s\in 2^n\}$ are distinct.  To see that, let $r\neq s$ in $2^n$. Thus, there is $k<n$ and  $t\in 2^{k}$ such that $\Delta(s,r) = t$. Then, $s(k)\neq r(k)$. Assume without loss of generality that $s(k) = 0$ and $r(k) = 1$. Then, $H(y_{  s})(m_t) = H(y_{t^\smallfrown 0})(m_t)  \neq H(y_{t^\smallfrown 1})(m_t)  = H(y_{r})(m_t).$

Now, we have defined for every $f\in 2^\omega$, a play of the game $$P_f = ({{\mathcal A}_{0, F_{f | 1}}}, x_{f| 1},\ldots, {{\mathcal A}_{{n}, F_{f| n+1}}}, x_{f| n+1 }, \ldots).$$  

If there is an $f\in 2^\omega$ such that 
$\bigcap_{k\in\omega} A_k^{i_{f| k+1}} = \emptyset,$
then the play of the game, corresponding to that $f$, is the one when Player \textbf{I} defeats $\sigma$. Otherwise, for each $f\in 2^\omega$ there is a unique $  y_f$ such that $$\bigcap_{k\in\omega} A_{k, F_{f|(k+1)}}^{i_{f| k+1}} = \{{  y_f}\}.$$ 

The uniqueness follows from the fact that if $y$ is in the preceding intersection, then $H(y)=(i_{f|k+1})_{k \in \omega}$, as $H$ is injective. Hence, the mapping $G:2^\omega\to H(Y)$ is one-to-one, where  $G(f) = H(y_{   f})$, for each $f\in 2^\omega$. It is also continuous. Indeed, let $V$ be a basic open set in $H(Y)$, then there exist $n$ and $s\in 2^n$ such that $V = \{f\in 2^\omega: f \supset s\}\cap H(Y)$.   Note that  $H(y_f)(k) = g_{f| k+1}(k)$, for all  $f\in G^{-1}(V)$. Then, for every $k<n$, $s(k) = g_{f|k+1}(k)$ for all  $f\in G^{-1}(V)$. Therefore, there exists $t_s\in 2^n$ such that $s(k) = g_{t_s| k+1}$.  Hence, $G^{-1}(V) = \{f\in 2^\omega: f \supset t_s\}$ which is a basic open set in $2^\omega$.  Then $H(Y)$ contains a Cantor set which is a contradiction.
\end{proof}
\section{Anti-Luzin and nowhere mad families}

    In the previous section, we improved results from \cite{AS} which implied the existence of non-normal semi-proximal Isbell-Mrówka spaces in ZFC. In this section, we investigate how close they are to being normal by showing that semi-proximal almost disjoint families are nowhere MAD and anti-Luzin. It is well-known that normal Isbell-Mrówka spaces also have these properties (see \cite{Hrusak2014}).

    A Luzin family is an almost disjoint family of cardinality $\omega_1$ which can be indexed as $(a_\alpha:\alpha<\omega_1)$ so that for every finite $F\subseteq N$ and $\alpha<\omega_1$, the set $\{\beta<\alpha: a_\beta\cap a_\alpha\subseteq F\}$ is finite. These families have the property that no two disjoint uncountable sets can be separated by open sets. To better study these families, the following definition was introduced in \cite{roitman1998luzin}:
        \begin{definition}
            An almost disjoint family $\mathcal A$ is said to be anti-Luzin if for every pair of uncountable subsets of $\mathcal A$, $\mathcal A_0, A_1$, there exists uncountable sets $\mathcal B_0\subseteq \mathcal A_0, \mathcal B_1\subseteq \mathcal A_1$ which can be separated by open sets.
        \end{definition}

If $\mathcal A$ is an almost disjoint family, $\mathcal B, \mathcal C\subseteq A$ and $X\subseteq \omega$, we say that \textit{$X$ separates $\mathcal B$ and $\mathcal C$} if for every $b \in \mathcal B$, $b\subseteq^* X$ and for every $c \in \mathcal C$, $b\cap X=^*\emptyset$. In case such an $X$ exists, we say $\mathcal B$, $\mathcal C$ can be separated. An uncountable almost disjoint family $\mathcal A$ is said to be \textit{inseparable} if no two uncountable disjoint subsets can be separated. Note that $\mathcal B, \mathcal C$ can be separated if and only if there exist two disjoint open sets $G, G'\subseteq \Psi(\mathcal A)$ such that $\mathcal B\subseteq G$, $\mathcal B\subseteq G'$. It is well known that every Luzin family is inseparable, but under CH there exists an inseparable family with no Luzin subfamily \cite{AbrahamShelah}.

\begin{proposition}\label{inseparable}
   Semi-proximal almost disjoint families are anti-Luzin.
\end{proposition}
\begin{proof}
    Let $\mathcal A$ be an almost disjoint family which is not anti-Luzin. There exists $\mathcal A'\subseteq \mathcal A$ such that no uncountable subsets of $\mathcal A'$ can be separated. Let $\mathbb U$ be a uniform structure for $\Psi(\mathcal A)$, and $\mathbb B$ be the basis of all symmetric open elements of $\mathbb U$.

    We show that \textbf{II} has a winning strategy $\sigma$ for the proximal game for $\mathbb B$. The strategy $\sigma$ will satisfy the following properties for all $k\in \omega$ and all decreasing sequences $(U_0, \dots, U_k)\in \mathbb B^{k+1}$:

    \begin{enumerate}[1)]
         \item $\sigma(U_0, \dots, U_k) \in \mathcal A'$.
         \item There exists $V\in \mathbb B$ such that $2V\subseteq U_k$ and that $V[\sigma(U_0, \dots, U_k)]\cap \mathcal A'$ is uncountable.
         \item $\sigma(U_0, \dots, U_{k-1})\neq \sigma(U_0, \dots, U_{k})$ if $k>0$.
    \end{enumerate}

    We show how to define such a $\sigma$ recursively.
    
    For $U \in \mathbb B$, fix $W\in \mathcal U$ such that $4 W\subseteq U$. $\{W[n]: n \in \omega\}$ covers $\mathcal A'$, so there exists $n \in \omega$ such that $W[n]\cap \mathcal A'$ is uncountable. Let $V=2W$. Let $\sigma(U)\in \mathcal A'\cap W[n]$. Then $V[\sigma(U)]\cap A'$ is uncountable as $W[n]\cap \mathcal A'\subseteq V[\sigma(U)]\cap \mathcal A'$.

    Assume $k\geq 0$ and that we have defined $\sigma(U_0, \dots, U_k)$, where $(U_0, \dots, U_k)$ is a decreasing sequence in $\mathbb B$. Let $U_{k+1}\in \mathbb B$ be such that $\mathcal U_{k+1}\subseteq U_k$. We must define $\sigma(U_0, \dots, U_{k+1})$.
    
    Let $W \in \mathcal U$ be such that $4W\subseteq U_{k+1}$ and $V=2W$. As $U_{k}[\sigma(U_0, \dots, U_k)]\cap \mathcal A'$ is uncountable, there exists $n\in \omega$ such that $W[n]\cap U_k[\sigma(U_0, \dots, U_k)]\cap \mathcal A'$ is uncountable.
    Let     $\sigma(U_0, \dots, U_k, U_{k+1}) \in W[n]\cap U_k[\sigma(U_0, \dots, U_k)]\cap \mathcal A'\setminus \{\sigma(U_0, \dots, U_k)\}$. Then $V[\sigma(U_0, \dots, U_k, U_{k+1})]\cap \mathcal A'\supseteq W[n]\cap \mathcal A'$ is uncountable.

    Thus, a strategy satisfying 1) and 2) exists. We show that $\sigma$ is a winning strategy: Assume $(U_0, x_0, U_1, x_1, \dots)$ is a play of the game according to $\sigma$. Then $(x_k)_{k \in \omega}$ does not converge as it is a non eventually constant sequence in the discrete closed set $\mathcal A$. It remains to see then that $\bigcap_{k \in \omega} U_k[x_k]\neq \emptyset$. It suffices to see that for each $k \in \omega$, $\mathcal A'\setminus U_k[x_k]$ is countable. Fix $k$. Fix $V$ as in 2). $V[x_k]\cap \mathcal A'$ is a uncountable subset of $\mathcal A'$ separated from $\mathcal A'\setminus \cl V[x_k]$, thus $\mathcal A'\setminus \cl V[x_k]$ is countable. As $\cl V[x_k]\subseteq 2 V[x_k]\subseteq U_k [x_k]$, we are done.
\end{proof}
In particular, Luzin families are not semi-proximal.

We say that an almost disjoint family $\mathcal A$ is \textit{nowhere MAD} iff there exists $X \in \mathcal I^+(\mathcal A)$ such that for every $B \in [X]^\omega$ there exists $a \in \mathcal A$ such that $|a\cap B|=\omega$.

\begin{proposition}
    Semi-proximal almost disjoint families are nowhere MAD.
\end{proposition}

\begin{proof}
    Assume $\mathcal A$ is not nowhere MAD. There exists $X \in \mathcal I^+(\mathcal A)$ such that for every $B \in [X]^\omega$ there exists $a \in \mathcal A$ such that $|a\cap B|=\omega$ - then $\mathcal B$ is uncountable.
    
    We show that for every uniformity $\mathbb U$ compatible with the topology of $\Psi(\mathcal A)$, \textbf{II} has a winning strategy for the proximal game on $\mathbb B$, the collection of all open symmetric entourages of $\mathbb U$. Recursively, define $V: \mathbb B^{<\omega}\setminus \{\emptyset\}\rightarrow \mathbb B$ such that, for every $k\geq 0$ and $(U_0, \dots, U_k) \in \mathbb B^{k+1}$:

    \begin{enumerate}[a)]
    \item $2V(U_0, \dots, U_k) \subseteq U_k$ 
        \item If $k>0$, $V(U_0, \dots, U_k)\subseteq V(U_0, \dots, U_{k-1})$.
    \end{enumerate}

This may be achieved recursively by halving $U_{k+1}\cap V(U_0, \dots, U_k)$.

    We recursively construct the strategy $\sigma$ for \textbf{II} such that, for every $k\geq 0$ and every decreasing sequence$(U_0, \dots, U_k)\in \mathbb B^{k+1}$:

    \begin{enumerate}[1)]
         \item $\sigma(U_0, \dots, U_k) \in \mathcal B$.
         \item $V(U_0, \dots, U_k)[\sigma(U_0, \dots, U_k)]\cap \mathcal B$ is uncountable.  
        \item If $k>0$, $\sigma(U_0, \dots, U_{k})\neq \sigma(U_0, \dots, U_{k-1})\}$.

    \end{enumerate}
    
To do that, proceed exactly as in the previous proposition.

    To see that $\sigma$ is a winning strategy for \textbf{II}, assume $(U_0, x_0, U_1, x_1, \dots)$ is a play of the game according to $\sigma$. Then $(x_k)_{k \in \omega}$ does not converge as it is not eventually constant in discrete closed set $\mathcal B$. It remains to see then that $\bigcap_{k \in \omega} U_k[x_k]\neq \emptyset$. For each $k \in \omega$, let $V_k=V(U_0, \dots, U_k)$ and $B_k=V_k\cap X$.
    It follows that $\bigcap_{k \in \omega} \cl V_k[x_k]\subseteq \bigcap_{k \in \omega} 2 V_k[x_k]\subseteq \bigcap_{k \in \omega} U_k[x_k]$. Clearly, $(B_k: k \in \omega)$ is a decreasing sequence of natural numbers, and, given $k \in \omega$, $x_k\cap X\subseteq^* V_k[x_k]\cap X=B_k$, so $B_k$ is infinite. Let $P$ be a pseudointersection of this sequence contained in $X$. Let $a \in \mathcal B$ be such that $|a\cap P|=\omega$. Then $a \in \bigcap_{k \in \omega} \cl B_k\subseteq \bigcap_{k \in \omega} \cl V_k[x_k]\subseteq \bigcap_{k \in \omega} U_k[x_k]$.
    
\end{proof}
%%%%%%%%%%%%%%%%%%%%%%%%%%%%%%%%%%%%%
\section{\texorpdfstring{$\Psi$}{Psi}-like spaces from ladder systems}

In this section, we study semi-proximality in ladder system $\Psi$-spaces.

Recall that a $\clubsuit^*$ sequence is a ladder system (
$(L_\alpha: \alpha\in \Lim)$ so that for each uncountable $A\subseteq \omega_1$,
$$\{\alpha:|L_\alpha\cap A|=\aleph_0\}\text{ contains a club}. $$

The principle $\clubsuit^*$ follows from $\diamondsuit^*$, which holds in \textbf{L}, and was introduced in \cite{AS} where it was used to construct a de Caux-type Dowker space that is not semi-proximal. It is essentially a corollary to that proof that if ${\mathcal L}$ is a $\clubsuit^*$-sequence then $\Psi({\mathcal L})$ is not semi-proximal. In what follows, we provide a direct proof for this fact for the sake of completeness.

\begin{lemma}\label{unctble}  If the ladder system ${\mathcal L}$ satisfies $\clubsuit^*$, then
\begin{enumerate}[(a)]
    \item If $A\subseteq \omega_1\times\{0\}$ is uncountable, then $\overline{A}(1) = \{\alpha: \langle\alpha,1\rangle\in \overline{A}\}$ contains a club.
    \item  If $S$ is a stationary subset of $\omega_1$ and $U$ is any open neighbourhood of $S\times\{1\}$, then $U\cap(\omega_1\times\{0\})$ is co-countable.
    \item For every uniformity ${\mathbb U}$ on $\Psi(\mathcal{L})$, $U\in{\mathbb U}$, and $S$ a stationary subset of $\omega_1$, there exists a stationary set $S^\prime\subset S$ such that $(S^\prime\times\{1\})^2\subset U$. 
\end{enumerate}

\end{lemma} 
\begin{proof} (a) Let $A$ be uncountable in $\omega_1\times\{0\}$. Then $\{\alpha: |(L_\alpha\times \{0\})\cap A| = \omega\}$ contains a club $C$. Hence, $\langle\alpha,1\rangle\in \overline{A}$, for all $\alpha\in C$, so that $\overline{A}({1})$ contains a club.

(b) Let $S$ be a stationary subset of $\omega_1$ and $U$ be any open neighbourhood of $S\times\{1\}$. Suppose $A = (\omega_1\times\{0\})\setminus U$ is uncountable. Since $A(0)\in[\omega_1]^{\omega_1}$, there exists a club $C$ such that $A(0)\cap L_\alpha$ is infinite for all $\alpha\in C$. But for every $\alpha\in C\cap S$, $\langle\alpha, 1\rangle\in U$. Hence, $L_\alpha\times\{0\}\subset^* U$, contradiction. Thus, $U\cap(\omega_1\times\{0\})$ is co-countable.

(c) Let ${\mathbb U}$ be a uniformity on $X$,  $U\in{\mathbb U}$ and $S$ be a stationary subset of $\omega_1$. Since $S\times\{0\}$ is uncountable, $\overline{S\times\{0\}}(1)$ contains a club $C$, by Lemma \ref{unctble} (a).  Let $A_\alpha = V[\langle\alpha, 1\rangle]\cap (S\times\{0\})$, for each $\alpha\in C$, where $V\in\mathfrak{U}$ with $2V\subseteq U$. Then $A_\alpha \neq\emptyset$, for each $\alpha\in C$. Define $f: C \to S$ by $f(\alpha) = \min(A_\alpha(0))$.   Then by Fodor's Lemma, there exist $\beta\in S$ and a stationary subset $S^\prime\subset S$ such that $\beta= f(\alpha)$, for all $\alpha\in S^\prime$. Hence $S^\prime\times\{1\}\subset V[\langle\beta,0\rangle]$. Thus, for every $\alpha,\gamma\in S'$, $\langle\gamma,1\rangle, \langle\alpha,1\rangle\in U$. Therefore, $(S^\prime\times\{1\})^2\subset U$.   \end{proof}

We are now ready to prove the result:
\begin{theorem}
    If the ladder system ${\mathcal L}$ is a $\clubsuit^*$ sequence, then $\Psi({\mathcal L})$ is not semi-proximal.
\end{theorem}

\begin{proof}
    Let ${\mathbb U}$ be any uniformity which generates the topology of $\Psi({\mathcal L})$. We will define a winning strategy $\sigma$ for Player II in the proximal game on $(\Psi({\mathcal L}),{\mathbb U})$. In inning $0$, Player I plays $U_0\in{\mathbb U}$. Then, by Lemma \ref{unctble} (c), there exists a stationary set $S_0\subset \omega_1$ such that $(S_0\times\{1\})^2\subset U$. Let $\sigma(U_0)\in S_0\times\{1\}$. In inning $n+1$, suppose that $\sigma(U_0,\ldots,U_n)$ has been defined so that there is a stationary set $S_n$ such that  $(S_n\times\{1\})^2\subset U_n$ and 
 $\sigma(U_0,\ldots,U_n) \in S_n\times\{1\}$.

Let $U_{n+1}\subset U_{n}$. Since $S_{n}$ is stationary, there exists a stationary set $S_{n+1}\subset S_{n}$ such that $(S_{n+1}\times\{1\})^2\subset U_{n+1}.$  Define $\sigma(U_0,\ldots,U_{n+1}) \in S_{n+1}\times\{1\}$. Note that  $(\sigma(U_0,\ldots,U_n) : n\in\omega)$ is not convergent in $\Psi({\mathcal L})$ as it is subset of a closed discrete subspace $\text{Lim}\times\{1\}$.  Since $S_{n}\times\{1\} \subset U_n[ \sigma(U_0,\ldots,U_n) ]$ for all $n$, then by Lemma \ref{unctble} (b),  $U_n[ \sigma(U_0,\ldots,U_n) ]\cap (\omega_1\times\{0\})$ is co-countable, for all $n$, thus $\bigcap_{n\in\omega} U_n[ \sigma(U_0,\ldots,U_n) ]\neq\emptyset$.
\end{proof}

We turn now to the question of whether $\Psi({\mathcal L})$ can be semi-proximal for some ladder system on $\omega_1$. Recall that ${\mathcal L}$ is uniformizable, if for each sequence of functions $f_\alpha:L_\alpha\rightarrow \omega$ there is a function $F:\omega_1\rightarrow \omega$ such that for all $\alpha$, $F|_{L_\alpha}=^* f_\alpha$ \cite{Shelah1980}.  Shelah introduced this notion and proved that not only does MA+$\neg$CH imply that all ladder systems are uniformizable but it is consistent with CH that all ladder systems are uniformizableENG. We proceed to show that every uniformizable ladder system is semi-proximal.

\begin{proposition}\label{unifladder}
    Every uniformizable ladder system is semi-proximal.
    \end{proposition}
    
    \begin{proof}

     Consider $Z = \{x_\alpha : \alpha\in \text{Lim}\}\subseteq 2^\omega$ enumerated injectively, and define $f_\alpha: L_\alpha\to 2^{<\omega}$ by $f_\alpha(\beta_\alpha^n) = x_\alpha|_n\in 2^n$, where $\beta_\alpha^n$ is the $n$-th element of $L_\alpha$. Since ${\mathcal L}$ is uniformizable, there exists $F: \omega_1\to 2^{<\omega}$ such that $F|_{ L_\alpha}=^*f_\alpha$. 
    Define $H:\Psi({\mathcal L})\rightarrow 2^{\leq \omega}$ by $H((\alpha,0))=F(\alpha)$ and $H((\alpha,1))=x_\alpha$.     Let $X=\Psi(\mathcal L)$.      For each $i \in 2$ and $k \in \omega$, let $$A_k^i=\{x \in X: |H(x)|\geq (k+1) \andd H(x)(k)=i\}.$$ Note that each $A_k^i$ is clopen. Now, for each finite $F \subseteq \Lim$ and $k \in \omega$, fix $G_{F,k}\subseteq \omega_1$ finite so that:
    \begin{itemize}
        \item For every $\alpha \in F$, $|H(L_\alpha\setminus G_{F,k})|=f_\alpha(L_\alpha\setminus G_{F,k})$.
        \item $\mathcal B_{F, k}=\{\{(\alpha, 1)\}\cup (L_\alpha\setminus G_{F,k})\times \{0\}: \alpha \in F\}\cup\{\{x\}:x \in X,\, |H(x)|\leq n\}\cup \{\{x\}:x \in G_{F, k}\}$ is pairwise disjoint - for clarity, $\mathcal B_{k,\emptyset}=\emptyset$.
    \item For every $\alpha \in F$, $f_\alpha[L_\alpha\setminus G_{F,k}]\cap 2^{\leq n}=\emptyset$.
    \end{itemize}

    Let $A_{k,F}^i = A_k^i\setminus \bigcup \mathcal B_{k,F}$ and define $\mathcal A_{k,F} =\{A_{k,F}^i: i<2\}\cup \mathcal B_{k,F}$. Note that, each $\mathcal A_{k,F}$ is a clopen partition of $X$. Let $\mathbb U$ be the uniformity generated by this family of clopen partitions. This uniformity is easily seen to generate the topology of $\Psi(\mathcal L)$. If we choose $Z$ to have no copy of the Cantor Set, Lemma \ref{technicalMachinery} shows that $\mathbb U$ is semi-proximal.
    
\end{proof}

\begin{corollary}
    The existence of a non semi-proximal ladder system is independent of ZFC.
\end{corollary}

In Section 2, we provided a sufficient condition for the semi-proximality of $\mathbb R$-embeddable almost disjoint families. Note that the definition of ${\mathbb R}$-embeddable given in section 2 has a natural formulation for ladder systems:

\begin{definition} A ladder system $\mathcal L$ is ${\mathbb R}$-embeddable if there is a continuous $f:\Psi({\mathcal L})\rightarrow {\mathbb R}$ such that $f\upharpoonright \mathcal{L}$ is injective. 

\end{definition}

The first lines of Proposition \ref{unifladder} shows the following:

\begin{proposition}
If ${\mathcal L}$ is uniformizable, then ${\mathcal L}$ is $\mathbb{R}$-embeddable. 
\end{proposition}

We conjecture that if $\mathcal L$ is a ladder system and there is an $H:\Psi(\mathcal L)\rightarrow X$ witnessesing ${\mathbb R}$-embeddability so that the range of the $H|_{\Lim\times \{1\}}$ contains no copy of the Cantor set, then $\mathcal L$ is semi-proximal. Proposition \ref{unifladder} showed that if the embedding is particularly nice then it is semi-proximal. Also by Proposition \ref{unifladder} we can note the converse is at least consistently not true: It is consistent with CH that all ladder systems on $\omega_1$ are uniformizable and moreover the proof showed that it can be mapped onto any family of $\omega_1$ many branches. So it is consistent with CH that every ladder system is semi-proximal and can embedded onto the entire family of branches in $2^\omega$. This observation should be compared to Problem \ref{2.6} which asked if there is any almost disjoint with this property. 

\begin{problem}
    Is $\mathcal L$-semiproximal whenever $\Psi(\mathcal L)$ is $\mathbb R$-embeddable?
\end{problem}

\section*{Acknowledgements}

The second author is supported by FAPESP (Fundação de Amparo à Pesquisa do Estado de São Paulo), process number 2023/17856-7.

The third author is supported by NSERC grant 53097.

\end{document}